\documentclass[runningheads]{llncs}

\usepackage{amsmath, amsfonts, amssymb, graphicx, xcolor, dsfont, enumerate, dcolumn,subcaption}
\usepackage{arydshln}
\usepackage[numbers,square, comma, compress]{natbib}
\usepackage{algorithm}
\usepackage{algorithmic}
\usepackage{makecell}
\usepackage{ifthen}
\usepackage{threeparttable}

\let\doendproof\endproof
\renewcommand\endproof{\hfill\qed\doendproof}

\usepackage{hyperref}
\usepackage{xurl}
\usepackage[capitalise]{cleveref}
\PassOptionsToPackage{hyphens}{url}\usepackage{hyperref}

\usepackage{tikz}

\usepackage{datetime}

\newtheorem{observation}{Observation}

\begin{document}

\title{On graphs isomorphic with their conduction graph}

\titlerunning{On graphs isomorphic with their conduction graph}
% If the paper title is too long for the running head, you can set
% an abbreviated paper title here
%

\author{Aidan Birkinshaw\inst{1}, Patrick W. Fowler\inst{1}, Jan Goedgebeur\inst{2,3} \and
Jorik Jooken\inst{2}\thanks{Corresponding author}}

\authorrunning{Birkinshaw et al.}
% First names are abbreviated in the running head.
% If there are more than two authors, 'et al.' is used.

\institute{
Department of Chemistry, University of Sheffield, Sheffield S3 7HF, UK\\
\email{a.birkinshaw5@gmail.com,p.w.fowler@sheffield.ac.uk}
\and
Department of Computer Science, KU Leuven Kulak, 8500 Kortrijk, Belgium\\
\email{jan.goedgebeur@kuleuven.be,jorik.jooken@kuleuven.be}
\and
Department of Applied Mathematics, Computer Science and Statistics, Ghent University, 9000 Ghent, Belgium\\
}

\maketitle

\begin{abstract}
Conduction graphs are defined here in order
to elucidate at a glance the often complicated 
conduction behaviour of molecular graphs as 
ballistic molecular conductors. The graph $G\sp{\mathrm C}$ 
describes all possible conducting devices associated with a given base graph $G$ within the context of the Source-and-Sink-Potential model of ballistic conduction. The graphs $G\sp{\mathrm C}$ and $G$ have the same vertex set, and each edge $xy$ in $G\sp{\mathrm C}$ represents a conducting device with graph $G$ and connections $x$ and $y$ that conducts at the Fermi level. If $G\sp{\mathrm C}$ is isomorphic with the simple graph $G$ (in which case we call $G$ conduction-isomorphic), then $G$ has nullity $\eta(G)=0$ and is an ipso omni-insulator. Motivated by this, examples are provided of ipso omni-insulators of odd order, thereby answering a recent
question. 
For $\eta(G)=0$, $G\sp{\mathrm C}$ is obtained by {\lq booleanising\rq} the inverse adjacency matrix $A^{-1}(G)$,
to form $A(G\sp{\mathrm C})$, i.e.\ by replacing all non-zero entries
$(A(G)\sp{-1})_{xy}$ in
the inverse by $1+\delta_{xy}$ where $\delta_{xy}$ is the Kronecker delta function. Constructions 
of conduction-isomorphic graphs are given for the cases of $G$ with minimum degree equal to two or any odd integer. Moreover, it is shown that given any connected non-bipartite conduction-isomorphic graph $G$, a larger conduction-isomorphic graph $G'$ with twice as many vertices and edges can be constructed. It is also
shown that there are no 3-regular conduction-isomorphic graphs. 
A census of small (order $\leq 11$) connected conduction-isomorphic graphs and small (order $\leq 22$) connected conduction-isomorphic graphs with maximum degree at most three is given. For $\eta(G)=1$, it is shown that $G\sp{\mathrm C}$ is connected if and only if $G$ is a nut graph (a singular graph of nullity one that has a full kernel vector).

%\bigskip\noindent \textbf{MSC 2020:} 05C45, 05C85 
\end{abstract}

%{\hskip2truecm Version of {\today,\: \currenttime}}
%\vspace{-5mm}

\section{Introduction}
The ballistic conduction of electrons through a molecular $\pi$ system that is connected by leads into a circuit can be modelled by the Source-and-Sink-Potential (SSP) approach~\cite{E11,GEZ07}, which itself can be pared down to a purely graph-theoretical form~\cite{PF08}.
\subsection{The graph theoretical model of molecular conduction}
The critical quantity is $T(E)$, the fractional transmission of electrons with energy $E$ through molecular graph $G$ connected via vertex $l$ to a source $L$ representing the incoming left lead, and via vertex $r$ to a sink $R$ representing the outgoing lead.
The triple $(G, l, r)$ defines a {\it device}. We are usually most interested in transmission of a device at energy $E = 0$, the Fermi or non-bonding level.
In the graph-theoretical version of the SSP, $T(0)$ is a function of four characteristic polynomials $s$, $t$, $u$, $v$ which are respectively $\varphi(G)$, $\varphi(G-l)$, $\varphi(G-r)$, $\varphi(G-l-r)$~\cite{PF08}. For a symmetric device, i.e.\ one where both leads are of the same material, $T(0)$ is given by evaluation in the limit $E \rightarrow 0$  of the function
%$$ T = 
%\frac{4(ut-sv) \beta\tilde\sp2}{(s-v \beta\tilde\sp2))\sp2 
%+ (t + u)\sp2\beta\tilde\sp2},$$

$$ T = \frac{4(ut-sv) \tilde{\beta}^2}{(s-v \tilde{\beta}^2)^2 + (t + u)^2 \tilde{\beta}^2},$$
where the numerator polynomial $(ut-sv)$ is a square $j^2$ (by Jacobi’s Theorem~\cite{S51}) and $\tilde{\beta}^2$ is a parameter describing relative strengths of interactions within the molecule, between molecule and leads, and within the leads~\cite{FPTBS14}. 

Devices may be {\it distinct} ($l \neq r$) or \textit{ipso} ($l = r$), and selection rules based on Cauchy’s Interlacing Theorem have been developed to predict whether any given device conducts ($T(0) \neq 0$) or insulates ($T(0) = 0$) at the Fermi level~\cite{FPTM09}. These are based on the {\it nullity signature} of the device, $(\eta(G), \eta(G-l), \eta(G-r), \eta(G-l-r))$ where $\eta(H)$ is the multiplicity of the eigenvalue zero in the spectrum of graph $H$. This signature suffices to distinguish conduction and insulation in all but one case. The sole
exception is when $\eta(G) = \eta(G-l) = \eta(G-r) = \eta(G-l-r)$, and it is then necessary to consider also the number of zero roots of $j\sp2$ to check for ‘accidental’ zeroes arising from the vanishing of $j\sp2 = (ut-sv)/E\sp{\eta(G)}$.

Consideration of the selection rules allows a definition of {\it classes} of conductors based on cases where all devices in a given class (e.g.\ distinct, ipso or both (= strong)) have some uniform behaviour.  This gives rise to the notions of omni-conductors and omni-insulators, and their near-omni and $d$-omni specialisations~\cite{FBPS20,FPTBS14,FSBSP17}.  Here, we explore a different approach that gives a global representation of {\it all possible} devices derived from a given base graph $G$.  This is achieved by defining a
conduction graph, $G\sp{\mathrm C}$, to be associated with the base graph $G$.

\subsection{The conduction graph}
The set of devices $(G, l, r)$ defines a {\it conduction matrix}, $C(G)$, indexed by device connection vertices, and
containing as entry on row $l$ and column $r$ the value of $T(0)$ for the specific device. This matrix can be used to construct a graph, as follows.

First, for any matrix $A \in \mathbb{R}^{n \times m}$, we define $booleanise(A)$ as the matrix obtained by replacing all off-diagonal non-zero entries of $A$ by 1 and all diagonal non-zero entries of $A$ by 2.

From the conduction matrix for the devices derived from $G$, a new matrix is constructed by assigning the element in row $l$ and column $r$ of a new matrix as $1 + \delta_{lr}$ (where $\delta_{ij}$ is the Kronecker delta function) if and only if 
the device $(G, l, r)$ conducts. By the above definition, this 
new matrix is just $booleanise(C(G))$. The new matrix is then interpreted as the adjacency matrix\footnote{We adopt the convention that a loop is represented by the entry 2 in the adjacency matrix.} of a graph $G\sp{\mathrm C}$, i.e.
$$A(G\sp{\mathrm C}) = booleanise(C(G)).$$  
\par\noindent
In plain language, each conducting distinct device $(G, l, r)$ leads to an edge in $G\sp{\mathrm C}$, and each conducting ipso device leads to a loop in $G\sp{\mathrm C}$.

The definition of the graph
$G\sp{\mathrm C}$ is intended to embody a transparent summary of the varied conduction behaviour of the underlying graph $G$.  An appealing feature of this description is that, since $G\sp{\mathrm C}$ is itself a graph, it is amenable to the standard techniques of graph theory.  It also suggests natural questions about the relationship of $G\sp{\mathrm C}$ to $G$, and about
the extent to which $G\sp{\mathrm C}$ inherits or surpasses properties of $G$.

\subsection{Conduction-isomorphic graphs}
 Two graphs $G=(V,E)$ and $G'=(V',E')$ are \textit{isomorphic} if there exists a bijection $h : V(G) \rightarrow V(G')$ between the vertex sets of $G$ and $G'$ such that any two vertices $u_1, u_2 \in V(G)$ are adjacent in $G$ if and only if $h(u_1)$ and $h(u_2)$ are adjacent in $G'$. We call a graph $G$ \textit{conduction-isomorphic} if $G$ is isomorphic with its conduction graph $G^\text{C}$. This notion is the central topic of this paper. 
 
 Conduction-isomorphic graphs are interesting, because for some applications it is necessary to reason about properties of the conduction graph of a graph $G$. In general, the conduction graph of $G$ is a non-trivial function of $G$, which can make this difficult, but for conduction-isomorphic graphs the task is easier, as the conduction graph of $G$ inherits many properties from $G$.

 \section{Outline}
The remainder of this paper is structured as follows: in Section~\ref{sec:prelim} we discuss preliminaries about the graph-theoretic model of conduction and introduce definitions that will be used throughout the paper. In Section~\ref{sec:ipsoOmniInsulators} we show how conduction-isomorphic graphs are related to ipso omni-insulators and answer a question of Fowler et al.~\cite{FPTBS14} by showing the existence of ipso omni-insulators of odd order. Next, in Section~\ref{sec:infFam} we demonstrate the existence of several infinite families of conduction-isomorphic graphs. We complement these results in Section~\ref{sec:condIsoFromCondIso} by showing that for any given connected non-bipartite conduction-isomorphic graph $G$ it is possible to construct a larger conduction-isomorphic graph $G'$ with twice as many vertices and edges. On the other hand, we show in Section~\ref{sec:no3Regular} that no simple graphs in which every vertex has degree 3 (e.g.\ fullerenes form an important class of such graphs) can be conduction-isomorphic. In Section~\ref{sec:computationalResults} we discuss extensive computations that allow us to compile an exhaustive list of connected conduction-isomorphic graphs with at most 11 vertices and connected conduction-isomorphic graphs with maximum degree at most 3 and at most 22 vertices. Finally, in Section~\ref{sec:conclusion} progress is reviewed and some directions 
for further exploration are highlighted.

\section{Preliminaries and notation}
\label{sec:prelim}
The \textit{spectrum} $\sigma(G)=\{\lambda_1,\lambda_2,...,\lambda_n\}$ of a graph $G$ is the multiset of the eigenvalues of its adjacency matrix, $A(G)$, and the corresponding eigenvectors are the set of column vectors $\{x_1, x_2, ..., x_n\}$. An eigenvector corresponding to eigenvalue 0 is called a \textit{kernel vector}. We denote the element on row $u$ of column vector $x_i$ by $x_{i,u}$. The \textit{nullity} of the graph $G$ is equal to the multiplicity of the eigenvalue 0 in its spectrum. As noted earlier, the nullity of a graph $G$ plays an essential role in determining whether a device 
$(G, x, y)$ 
formed by connection of leads to vertices $x$ and $y$ of $G$
conducts or insulates at zero energy of the incoming electron.

We recall Cauchy's interlacing theorem~\cite{C93,H04}:
\begin{theorem}[\cite{C93,H04}]
    Let $G$ be a graph on $n$ vertices with eigenvalues $\lambda_1 \geq \lambda_2 \geq ... \geq \lambda_n$ and let $G\sp\prime$ be a graph with eigenvalues $\mu_1 \geq \mu_2 \geq ... \geq \mu_{n-1}$ on $n-1$ vertices obtained by removing a vertex from $G$. The spectra of $G$ and $G\sp\prime$ interlace, i.e.\ we have:
    $$\lambda_1 \geq \mu_1 \geq \lambda_2 \geq \mu_2 \geq ... \geq \mu_{n-1} \geq \lambda_n.$$
\end{theorem}
As a direct corollary of this theorem, we see that the nullity
can change by at most one on removal of a vertex:
\begin{corollary}
    Let $G$ be a graph on $n$ vertices with nullity $\eta$. The nullity of any graph $G'$ obtained by removing a vertex from $G$ is in the set $\{\eta-1,\eta,\eta+1\}$.
\end{corollary}

A vertex $u$ of a graph $G$ for which there is a kernel vector $x_i$ such that $x_{i,u} \neq 0$ is called a \textit{core vertex}. For such a vertex, we have $\eta(G-u)=\eta(G)-1$. All other vertices are called \textit{core-forbidden vertices}. Within the class of core-forbidden vertices, we also make a distinction between vertices whose removal leads to a graph $G'$ with nullity $\eta$ (\textit{middle vertices}) or $\eta+1$ (\textit{upper vertices}). Based on the nullities of the graphs, Fowler et al.~\cite{FPTBS14} obtained the \textit{selection rules} shown in Table~\ref{tab:selectionRules} to determine whether the device associated with vertices $u$ and $v$ of graph $G$ 
, i.e.\ the device $(G,u,v)$, conducts or not.

\begin{table}[htbp]
    \centering
    \begin{tabular}{|c| c | c | c | c || c |}
        \hline
        $u=v$ & Nullity of $G$ & Nullity of $G-u$ & Nullity of $G-v$ & Nullity of $G-u-v$ & \thead{Conduction?}\\
        \hline
        No & $\eta$ & $\eta+1$ & $\eta+1$ & $\eta+2$ & No\\
        No & $\eta$ & $\eta+1$ & $\eta+1$ & $\eta$ & Yes\\
        No & $\eta$ & $\eta+1$ & $\eta$ & $\eta+1$ & No\\
        No & $\eta$ & $\eta+1$ & $\eta$ & $\eta$ & Yes\\
        No & $\eta$ & $\eta+1$ & $\eta-1$ & $\eta$ & No\\
        No & $\eta$ & $\eta$ & $\eta$ & $\eta+1$ & Yes\\
        No & $\eta$ & $\eta$ & $\eta$ & $\eta$ & Maybe\\
        No & $\eta$ & $\eta$ & $\eta-1$ & $\eta-1$ & No\\
        No & $\eta$ & $\eta-1$ & $\eta-1$ & $\eta$ & Yes\\
        No & $\eta$ & $\eta-1$ & $\eta-1$ & $\eta-1$ & Yes\\
        No & $\eta$ & $\eta-1$ & $\eta-1$ & $\eta-2$ & No\\
        \hline
        Yes & $\eta$ & $\eta+1$ & $\eta+1$ & - & No\\
        Yes & $\eta$ & $\eta$ & $\eta$ & - & Yes\\
        Yes & $\eta$ & $\eta-1$ & $\eta-1$ & - & Yes\\
        \hline 
    \end{tabular}
    \vskip12pt
    \caption{The selection rules from~\cite{FPTBS14} that determine whether the device $(G,u,v)$ conducts or insulates at the Fermi level.}
    \label{tab:selectionRules}
\end{table}

We remark that the conduction behaviour of the device $(G,u,v)$ when all nullities of $G$, $G-u$, $G-v$ and $G-u-v$ are equal depends on the non-zero eigenvalues and the corresponding kernel vectors of $G$. More precisely, define 
\begin{equation*}
s_0 := 
\prod_{\substack{\lambda_i \in \sigma(G) \\ \lambda_i \neq 0}} -\lambda_i
\end{equation*}
and
\begin{equation*}
j_a := 
\sum_{\substack{\lambda_i \in \sigma(G) \\ \lambda_i \neq 0}} \frac{x_{i,u}x_{i,v}s_0}{-\lambda_i}.
\end{equation*}
In the case where the four nullities are equal, the device $(G,u,v)$ conducts if and only if $j_a \neq 0$~\cite{FPTBS14}.

To appreciate the conduction behaviours of all devices $(G,u,v)$ for all pairs of not necessarily distinct vertices $u,v$ of a graph $G$, we define the \textit{conduction graph $G\sp{\mathrm C}$ of $G$}. This graph has the same vertex set as $G$ and there is an edge between vertices $u$ and $v$ if and only if the device $(G,u,v)$ conducts. When we speak of a conduction graph $G\sp{\mathrm C}$ of a graph $G$, in this paper we will always assume that $G$ is connected and simple, i.e.\ does not contain parallel edges nor loops, whereas $G\sp{\mathrm C}$ might be disconnected or contain loops (but no parallel edges).

We now introduce a number of definitions that will be used throughout the paper. We call a square matrix $A$ \textit{orthogonal} if $A^T=A^{-1}$, we call it a \textit{permutation matrix} if every row and every column of $A$ contains precisely one non-zero entry equal to 1 and we call it a \textit{0-1 matrix} if every entry of $A$ is equal to 0 or 1. A \textit{canonical double cover} of a graph $G=(V,E)$ is a graph $G'=(V',E')$ with vertex set $V' = V \times \{0,1\}$ and edge set $E' = \{(u_1,j)(u_2,1-j)~|~u_1u_2 \in E, j \in \{0,1\}\}$. We call a simple graph $G$ \textit{$d$-regular} if every vertex in $G$ has degree $d$ (precisely $d$ neighbours).

\subsection{Conduction graphs and conduction classes}
\label{sec:gcandtla}

Various generic classes of device behaviour have been defined~\cite{FPTBS14}, and some of these correspond to special kinds of conduction graphs. The broadest classification is into omniconductors and omni-insulators of different types, which can be codified with two-letter acronyms, where the first letter describes the behaviour of the set of distinct devices and the second the ipso devices. The alphabet is $\{\texttt{C,I,X}\}$, denoting all conducting (\texttt{C}), all insulating (\texttt{I}), and mixed behaviour or nonexistence of the class (\texttt{X}). A number from $\{0,1,2\}$ is appended to indicate that $\eta(G)$ is 0, 1 or $\geq 2$, respectively. Of the 27 conceivable combinations, only 13 are realisable, with nullity playing a vital role in deciding whether a combination is allowed or not. For example, combinations \texttt{II} are impossible, combinations \texttt{CI} and \texttt{XI} are limited to nullity zero, and combinations with \texttt{I} as first letter can occur only for nullity 2 or more.

The codes have implications for the structure of $G\sp{\mathrm C}$. Clearly, a \texttt{C} in the first place implies that the vertices of $G\sp{\mathrm C}$ form a clique, an \texttt{I} in first place that they form an independent set. A \texttt{C} for second place implies that $G\sp{\mathrm C}$ has loops on all vertices, an \texttt{I} in second place implies that $G\sp{\mathrm C}$ is loopless. Hence, $G\sp{\mathrm C}$ is a complete graph decorated on all vertices with a loop for \texttt{CC0} and \texttt{CC1}; in the case of \texttt{CC1}, $G$ is a nut graph; for \texttt{CI0}, $G\sp{\mathrm C}$ is a complete graph. The impossible combination \texttt{II} would imply that $G\sp{\mathrm C}$ is the disjoint union of $n$ copies of $K_1$, a fact which was used implicitly in the proof that this 
{\lq strong omni-insulator\rq} combination is never realised.
The pure, i.e. {\texttt{X}}-free codes {\texttt{CC}}, {\texttt{CI}} and the unrealisable {\texttt{II}} imply adjacency matrices $A({G\sp{\mathrm C}})$
equal to $J + I$, $J - I$ and $\mathbf{0}$ , where $J$, $I$ and $\mathbf{0} $ are the all-one matrix, identity matrix and all-zeros matrix of appropriate order. If we want to stress the order $n$ when it is not clear from the context, we will write $J_n$, $I_n$ and $\mathbf{0}_n$.

A more detailed classification uses a three-letter acronym (TLA) and in its general form, the first letter describes behaviour of distinct devices where the connection vertices are separated by odd distance, the second letter describes the distinct devices with even distance between connections and the third letter applies to the set of ipso devices. Now there are 81 conceivable cases, of which 42 are impossible, 35 realisable and 4 undecided~\cite{FBPS20}, again with implications for the possible structure of $G\sp{\mathrm C}$.

For bipartite graphs, the TLA has a simpler description in that 
the three letters of the code refer to 
{\lq inter\rq}, 
{\lq intra\rq} and 
{\lq ipso\rq} devices~\cite{FSBSP17},
where connections are respectively, distinct and in different partite sets, 
distinct and in the same partite set, or not distinct.
In the bipartite case, only 14 combinations of letters and nullity are realisable~\cite{FSBSP17}. Four combinations are realisable by just one graph: $(K_2,K_1,K_{2,1},K_{3,3})$ for \texttt{CXI0}, \texttt{XXC1}, \texttt{ICX1}, \texttt{ICC2} with conduction graphs $K_2, K_1^{\mathrm{loop}}, K_1+K_2^{\mathrm{loop}}, 6K_1^{\text{loop}}$.
Here $G^{\mathrm{loop}}$ represents the graph $G$ where a loop is added to each vertex, and we have used 
$G+H$ to denote the disjoint union of two graphs $G$ and $H$, and $kG$ to denote the disjoint union of $k$ copies of $G$. 
Again {\texttt{X}}-free codes lead to simple forms for $A({G\sp{\mathrm C}})$. 
If vertices are ordered so that $A(G)$ is in block-diagonal form 
\[
\renewcommand{\arraystretch}{1.5}
 \left[
\begin{array}{c|c}
\mathbf{0} & B \\ \hline
B^T & \mathbf{0} \\
\end{array}
\right],
\]
then 
{\texttt{CII}}, {\texttt{ICC}} and {\texttt{IIC}} imply adjacency matrices
\[
\renewcommand{\arraystretch}{1.5}
 \left[
\begin{array}{c|c}
\mathbf{0} & J \\ \hline
J & \mathbf{0} \\
\end{array}
\right],
\]
\[
\renewcommand{\arraystretch}{1.5}
 \left[
\begin{array}{c|c}
J+I & \mathbf{0} \\ \hline
\mathbf{0} & J+I \\
\end{array}
\right],
\]
and
\[
\renewcommand{\arraystretch}{1.5}
 \left[
\begin{array}{c|c}
I & \mathbf{0} \\ \hline
\mathbf{0} & I\\
\end{array}
\right],
\]
respectively for $A(G^\textrm{C})$. TLA codes for various families of graphs of interest in chemistry are tabulated in~\cite{FSBSP17} (Table VI).

Catafused benzenoids are Kekulean and hence of even order and non-singular, with code \texttt{CII0}~\cite{FSBSP17}, so that for them $G\sp{\mathrm C}$ is loopless and contains edges for all odd distances (only) and hence is the complete bipartite graph $K_{n/2,n/2}$. Isomeric catafused benzenoids exist for hexagon counts $h>2$ and sets for fixed $h$ share the same conduction graph. Although the frameworks of isomeric
benzenoid base graphs are non-isomorphic, every bipartite graph with partite sets of equal size is a subgraph of $K_{n/2,n/2}$, and in particular $G$ for each isomer is a subgraph of the common $G\sp{\mathrm C}$.

Another graph class that has specific conduction characteristics is that of the uniform core graphs (UCG). A UCG is a core graph (all vertices are core) and has the extra property that all distinct devices have nullity signature $(\eta, \eta-1, \eta-1, \eta-2)$, which leads to insulation at the Fermi level (see Table~\ref{tab:selectionRules}). Hence UCGs have nullity of 2 or more, and code \texttt{IIC2}. Therefore $G\sp{\mathrm C}$ is isomorphic with $nK_1^{\text{loop}}$ for any UCG on $n$ vertices.

\subsection{Conduction graphs derived from non-singular graphs}
\label{sec:connonsing}

Fig.~{\ref{fig:allSmallGraphsAndTheirConductionGraph}}
shows the conduction graphs derived from the simple, connected graphs on up to at most 4 vertices. 
In general, each graph $G$ has a unique conduction graph $G^\text{C}$, but a given graph $H$ may be the conduction graph of two or more non-isomorphic parents. As the figure shows,
the cycle $C_4$ and the diamond ($K_4$ minus an edge) on $4$ vertices 
have isomorphic conduction graphs, which consist of 
the disjoint union of two copies
of $K_2\sp{\text{loop}}$.
The figure also shows two examples where $G^\text{C}$ is
isomorphic with $G$ ($K_2$ and $P_4$).
\begin{figure}[h!]
\begin{center}
\includegraphics[width=0.7\linewidth]{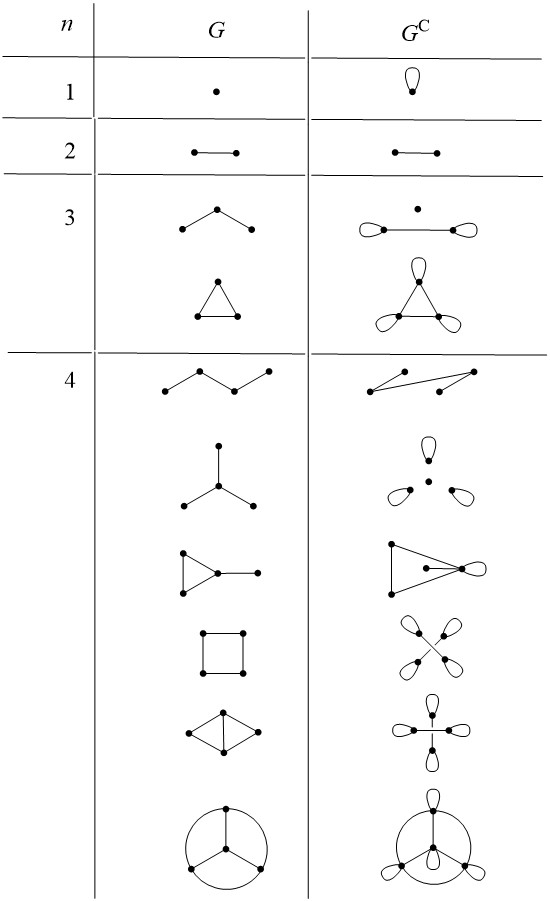} % Replace with your image path
\end{center}
\caption{An overview of all connected pairwise non-isomorphic graphs on at most 4 vertices and their conduction graphs.
}\label{fig:allSmallGraphsAndTheirConductionGraph} 
\end{figure}

For graphs with nullity 0, the conduction matrix has a particularly simple form, which gives an alternative way to describe the conduction graph:
\begin{theorem}[\cite{FPTBS14}]\label{th:booleanise}
Let $G$ be a graph of nullity 0. The adjacency matrix of the conduction graph $G\sp{\mathrm C}$ of $G$ is given by $booleanise(A^{-1})$.
\end{theorem}

We note that for any invertible matrix $A$, we have $(A^{-1})^{-1}=A$. For a graph $G$ of nullity 0, the conduction graph of the conduction graph of $G$, i.e.\ $(G\sp{\mathrm C})\sp{\mathrm C}$, may be isomorphic with $G$, as already shown in Fig.~\ref{fig:allSmallGraphsAndTheirConductionGraph} (numerical experiments suggest that this holds for all paths on an even number of vertices), but this is not always the case.

\subsection{Conduction graphs derived from singular graphs}
\label{sec:consing}

For a graph $G=(V,E)$ with nullity 1, we can write the adjacency matrix $A$ of $G$ as the following block matrix:
\[
\renewcommand{\arraystretch}{1.5}
A(G) = \left[
\begin{array}{c|c|c}
A_{LL} & A_{LM} & A_{LU} \\ \hline
A_{LM}^T & A_{MM} & A_{MU} \\ \hline
A_{LU}^T & A_{MU}^T & A_{UU}
\end{array}
\right]
\]

\noindent
Here, the letters L, M and U correspond to lower, middle and upper, and the blocks $A_{ij}$ represent the adjacencies between such vertices. Let $n_1$ be the number of lower vertices of $G$. From the selection rules in Table~\ref{tab:selectionRules}, the conduction graph $G\sp{\mathrm C}$ of $G$ has the following adjacency matrix:
    \[
A(G\sp{\mathrm C}) =\left[
\renewcommand{\arraystretch}{1.5}
\begin{array}{c|c|c}
J_{n_1}+I_{n_1} & \mathbf{0} & \mathbf{0}\\ \hline
\mathbf{0} & A_{MM}^{\text{Cond}} & A_{MU}^{\text{Cond}}\\ \hline
\mathbf{0} & (A_{MU}^{\text{Cond}})^T & A_{UU}^{\text{Cond}}
\end{array}
\right]
\]

Here, the matrices $A_{MM}^{\text{Cond}}, A_{MU}^{\text{Cond}}$ and $A_{UU}^{\text{Cond}}$ represent the conduction behaviour for devices associated with respectively middle and middle, middle and upper, and upper and upper vertices. Given
that every graph with nullity at least 1 has at least 1 core vertex,
we have that the conduction graph of a graph $G$ with nullity 1 is connected if and only if {\it all} vertices of $G$ are core vertices. Such graphs correspond precisely to~\textit{nut graphs}, a class of graphs which received attention in the literature. A nut graph is simple, connected, non-bipartite and has no leaves~\cite{SG98}, and it has the defining feature that it is singular, with nullity one and a full kernel vector. Nut graphs are exactly the strong omniconductors of nullity one~\cite{FPTBS14}. If the graph $K_1$ is excluded as trivial~\cite{SG98}, the smallest nut graphs have order 7. A census of small nut graphs, and statistics for nut graphs belonging to various special families are described in~\cite{CFG18} and are available online at~\url{https://houseofgraphs.org/meta-directory/nut}.

We call a graph $G$ a \textit{chemical graph} if $G$ is connected and the maximum degree of a vertex in $G$ is at most three (in the graph theory literature such graphs are called connected subcubic graphs). Chemical nut graphs have vertex degrees 2 and 3 only, and the smallest example is of order 9. The orders and degree signatures realisable by chemical nut graphs have been characterised~\cite{FPB21}. 
Several general constructions for obtaining large nut graphs from smaller ones are known~\cite{BFP24,GPS23,S08,SG98}.

The \textit{adjugate matrix} of a square matrix $A$ is the transpose of the cofactor matrix of $A$. We now recall the following fact about the adjugate matrix of a matrix with nullity one:
\begin{observation}[\cite{HJ13}]
\label{obs:adjugate}
Let $A$ be a square matrix of nullity 1. Then the adjugate matrix $adj(A)$ can be written as:
\begin{equation*}
    adj(A)=\alpha x x^T
\end{equation*}
where $\alpha$ is equal to the product of the non-zero eigenvalues of $A$, i.e.\ the pseudo-determinant, and $x$ is a kernel vector of $A$.
\end{observation}
From Observation~\ref{obs:adjugate}, the matrix obtained by booleanising the adjugate matrix of the adjacency matrix $A$ corresponding to a graph of nullity 1 has the following form:

    \[
booleanise(adj(A))=\left[
\renewcommand{\arraystretch}{1.5}
\begin{array}{c|c|c}
J_{n_1}+I_{n_1} & \mathbf{0} & \mathbf{0}\\ \hline
\mathbf{0} & \mathbf{0} & \mathbf{0}\\ \hline
\mathbf{0} & \mathbf{0} & \mathbf{0}\\
\end{array}
\right]
\]
Hence, whereas
for graphs $G$ of nullity 0 with adjacency matrix $A$, booleanising $A^{-1}$ yields the adjacency matrix of the conduction graph of $G$, but for graphs $G$ of nullity 1 other than nut graphs, booleanising $adj(A)$ reveals only a part of the conduction graph of $G$. 
For graphs with nullity larger than 1, the situation is more complicated and we do not know of any such analogous statements.

\section{Ipso omni-insulators and conduction-isomorphic graphs}
\label{sec:ipsoOmniInsulators}
An \textit{ipso omni-insulator} is a graph $G$ for which the conduction graph $G^\text{C}$ contains no loops. 
We recall the following important theorem about ipso omni-insulators:
\begin{theorem}[\cite{FPTBS14}]
 Every simple connected ipso omni-insulator has nullity 0.
\end{theorem}

Since in the current paper we are only interested in simple (i.e.\ without loops and parallel edges), connected graphs, we have 

\begin{observation}
\label{obs:condIsoIsIpsoOmniInsulator}
 Every simple, connected, conduction-isomorphic graph is an ipso omni-insulator.
\end{observation}
This is the main reason why we are interested in ipso omni-insulators in the present context. Furthermore:
\begin{corollary}
\label{cor:condIsoNullity0}
 Every simple, connected, conduction-isomorphic graph has nullity~0.
\end{corollary}

Fowler et al.~\cite{FPTBS14} determined all ipso omni-insulators in a number of different graph families, e.g.\ connected simple graphs on at most 10 vertices or chemical graphs on at most 16 vertices. Interestingly, while there are plenty of ipso omni-insulators of even order, the authors of that paper
were unable to find any ipso omni-insulators of odd order and asked whether such graphs cannot exist or whether they are just very rare. By searching through lists of graphs with at most two vertex orbits~\cite{HR20}, we found several ipso 
omni-insulators of odd order. We close this section by providing the smallest example that we found in Fig.~\ref{fig:ipsoOmniInsulator} (available at~\url{https://houseofgraphs.org/graphs/35457}).

\begin{figure}[h!]
\begin{center}
\begin{tikzpicture}[scale=0.85]
  \def\sides{15}
  \def\radius{3}
  \foreach \i in {1,...,\sides} {
    \fill ({360/\sides * \i}:\radius) circle (2pt);
  }
  
  % Draw the 15-gon
  \foreach \i in {1,4,7,10,13} {
    \draw ({360/\sides * (\i + 1)}:\radius) -- ({360/\sides * \i}:\radius);
    \draw ({360/\sides * (\i + 4)}:\radius) -- ({360/\sides * \i}:\radius);
    \draw ({360/\sides * (\i + 6)}:\radius) -- ({360/\sides * \i}:\radius);
    \draw ({360/\sides * (\i + 9)}:\radius) -- ({360/\sides * \i}:\radius);
    \draw ({360/\sides * (\i + 11)}:\radius) -- ({360/\sides * \i}:\radius);
    \draw ({360/\sides * (\i + 14)}:\radius) -- ({360/\sides * \i}:\radius);
  }
    \foreach \i in {2,5,8,11,14} {
    \draw ({360/\sides * (\i + 1)}:\radius) -- ({360/\sides * \i}:\radius);
    \draw ({360/\sides * (\i + 3)}:\radius) -- ({360/\sides * \i}:\radius);
    \draw ({360/\sides * (\i + 11)}:\radius) -- ({360/\sides * \i}:\radius);
    \draw ({360/\sides * (\i + 12)}:\radius) -- ({360/\sides * \i}:\radius);
    \draw ({360/\sides * (\i + 14)}:\radius) -- ({360/\sides * \i}:\radius);
  }
  \foreach \i in {3,6,9,12,15} {
    \draw ({360/\sides * (\i + 1)}:\radius) -- ({360/\sides * \i}:\radius);
    \draw ({360/\sides * (\i + 3)}:\radius) -- ({360/\sides * \i}:\radius);
    \draw ({360/\sides * (\i + 4)}:\radius) -- ({360/\sides * \i}:\radius);
    \draw ({360/\sides * (\i + 12)}:\radius) -- ({360/\sides * \i}:\radius);
    \draw ({360/\sides * (\i + 14)}:\radius) -- ({360/\sides * \i}:\radius);
  }
  
\end{tikzpicture}
\end{center}
% NkM``hQaQSSQbISRHe?
\caption{An ipso omni-insulator of order 15.}\label{fig:ipsoOmniInsulator} 
\end{figure}
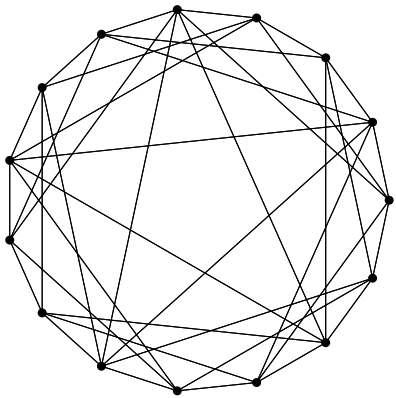

\par\noindent
None of the odd ipso omni-insulators that we found are chemical, and none are
conduction-isomorphic. We therefore pose the following questions:
\begin{question}
    Does there exist a chemical ipso omni-insulator of odd order?
\end{question}
\begin{question}
    Does there exist a conduction-isomorphic graph of odd order?
\end{question}
\section{Infinite families of conduction-isomorphic graphs}\label{sec:infFam}
In this section, 
we describe three infinite families of conduction-isomorphic graphs, each with a different minimum degree.

\subsection{Conduction-isomorphic graphs with minimum degree one}\label{sec:infFam1}
The first infinite family of graphs shows that any arbitrary graph occurs as an induced subgraph of infinitely many conduction-isomorphic graphs with minimum degree one:

\begin{theorem}\label{th:coronaGraphs}
Let $G$ be a graph on $n$ vertices. For all integers $k \geq 1$, there exists a conduction-isomorphic graph $G\sp\prime$ with minimum degree one on $2^kn$ vertices which contains $G$ as an induced subgraph.
\end{theorem}
\begin{proof}
Set $G_0 := G$. We will construct an infinite sequence of conduction-isomorphic graphs $G_1, G_2, ...$ with minimum degree one, where $G_i$ contains $G_{i-1}$ as an induced subgraph and $|V(G_i)|=2|V(G_{i-1})|$ ($i\geq 1$), from which the theorem follows.

Let $A$ be the adjacency matrix of $G_{i-1}$, which contains $2^{i-1}n$ vertices. Let $B$ be any square orthogonal 0-1 matrix with $2^{i-1}n$ rows such that $AB=BA$ and every row and column of $B$ contains precisely one entry equal to one (note that such a $B$ indeed always exists, for example by taking $B=I$). The adjacency matrix of $G_{i}$ is now given by:
    \[
\left[
\renewcommand{\arraystretch}{1.5}
\begin{array}{c|c}
A & B \\ \hline
B^T & \mathbf{0}
\end{array}
\right].
\]
Using blockwise matrix multiplication, the fact that $B$ is orthogonal and the fact that every orthogonal matrix is symmetric, we can see that the inverse of this matrix is given by:
    \[
\left[
\renewcommand{\arraystretch}{1.5}
\begin{array}{c|c}
\mathbf{0} & B \\ \hline
B^T & -A
\end{array}
\right].
\]

From Theorem~\ref{th:booleanise}, booleanising the above matrix yields the adjacency matrix of the conduction graph of $G_{i}$. Clearly $G_{i-1}$ occurs as an induced subgraph of $G_i$ by construction. Moreover, $G_{i}$ is isomorphic with its conduction graph with an isomorphism that maps every vertex of $V(G_i)$ corresponding to row $j$ to the vertex corresponding to row $j+2^{i-1}n$ and vice-versa (for all $0 \leq j \leq 2^{i-1}n-1$).
\end{proof}

We remark that the choice of $B$ in the previous proof is strongly constrained\footnote{A special case of conduction-isomorphic graphs arises when the adjacency matrix is equal to its inverse. This remark also shows that a matching, i.e.\ a forest of $K_2$s, is the only graph with that property.}:
\begin{remark}
Let $B$ be an orthogonal 0-1 matrix, then $B$ is a permutation matrix.
\end{remark}
\begin{proof}
Since every orthogonal matrix is symmetric and $BB^T=I$, we know that each row of $B$ can only contain at most one non-zero element. Moreover, $B$ cannot contain a row full of zeroes because $B$ is invertible, so $B$ must contain precisely one zero for each row. Hence, it is a permutation matrix.
\end{proof}

However, we present the proof of Theorem~\ref{th:coronaGraphs} in its general form, because it might be interesting for an edge-weighted variant of conduction-isomorphic graphs (where matrix entries can be distinct from 0 or 1) and where the matrix $B$ can contain more than one entry per row which is non-zero.

A \textit{corona graph} is obtained by adding a pendant edge to every vertex of a base graph $G$. The graphs considered in Theorem~\ref{th:coronaGraphs} are precisely the corona graphs (and thus every corona graph is conduction-isomorphic). Two infinite families of chemical conduction-isomorphic graphs which are corona graphs are obtained by choosing the base graph as a path (resulting in \textit{combs}) and a cycle (resulting in \textit{radialenes}). These chemically interesting graphs and their conduction graphs are shown in Fig.~\ref{fig:combsAndRadialenes}. Since every chemical graph is connected, we obtain that these graphs are the only chemical conduction-isomorphic graphs which are corona graphs (as a direct corollary of Theorem~\ref{th:coronaGraphs}):
\begin{corollary}
Let $G$ be a conduction-isomorphic corona graph, which is also a chemical graph. Then $G$ is a comb or a radialene. Moreover, every comb and radialene is conduction-isomorphic.
\end{corollary}

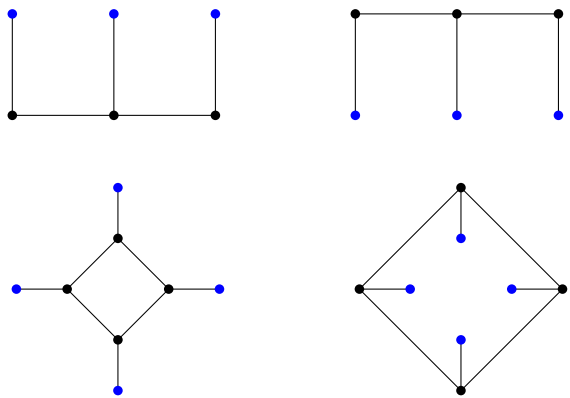
\begin{figure}[h!]

\begin{center}
\begin{tikzpicture}[scale=0.9]
  \def\inter{1.5}
  \def\k{3}

    \foreach \i in {2,...,\k}
    {
        \draw ({\inter * (\i-1)}, \inter * 1) -- ({\inter * (\i)}, \inter * 1);
    }

    \foreach \i in {1,...,\k}
    {
        \draw ({\inter * (\i)}, \inter * 1) -- ({\inter * (\i)}, \inter * 2);
    }

      \foreach \i in {1,...,\k}
    {
        \fill (\inter * \i,\inter * 1) circle (2pt);
    }

    \foreach \i in {1,...,\k}
    {
        \fill[blue] (\inter * \i,\inter * 2) circle (2pt);
    }
\end{tikzpicture} \quad\quad\quad\quad\quad%\quad
\vspace{0.5cm}
\begin{tikzpicture}[scale=0.9]
  \def\inter{1.5}
  \def\k{3}

    \foreach \i in {2,...,\k}
    {
        \draw ({\inter * (\i-1)}, \inter * 2) -- ({\inter * (\i)}, \inter * 2);
    }

    \foreach \i in {1,...,\k}
    {
        \draw ({\inter * (\i)}, \inter * 1) -- ({\inter * (\i)}, \inter * 2);
    }

      \foreach \i in {1,...,\k}
    {
        \fill[blue]  (\inter * \i,\inter * 1) circle (2pt);
    }

    \foreach \i in {1,...,\k}
    {
        \fill(\inter * \i,\inter * 2) circle (2pt);
    }

\end{tikzpicture} %\quad
\vspace{0.3cm}

\begin{tikzpicture}[scale=0.9]
  \def\sides{4}
  \def\radius{0.75}
  
  \foreach \i in {1,...,\sides} {
    \draw ({360/\sides * (\i + 1)}:\radius) -- ({360/\sides * \i}:\radius);
  }

  \foreach \i in {1,...,\sides} {
    \draw ({360/\sides * (\i)}:\radius) -- ({360/\sides * \i}:2*\radius);
  }

  \foreach \i in {1,...,\sides} {
    \fill ({360/\sides * \i}:\radius) circle (2pt);
  }

  \foreach \i in {1,...,\sides} {
    \fill[blue] ({360/\sides * \i}:2*\radius) circle (2pt);
  }
  
\end{tikzpicture} \quad\quad\quad\quad\quad
\begin{tikzpicture}[scale=0.9]
  \def\sides{4}
  \def\radius{0.75}
  
  \foreach \i in {1,...,\sides} {
    \draw ({360/\sides * (\i + 1)}:2*\radius) -- ({360/\sides * \i}:2*\radius);
  }

  \foreach \i in {1,...,\sides} {
    \draw ({360/\sides * (\i)}:\radius) -- ({360/\sides * \i}:2*\radius);
  }

    \foreach \i in {1,...,\sides} {
    \fill[blue]  ({360/\sides * \i}:\radius) circle (2pt);
  }

  \foreach \i in {1,...,\sides} {
    \fill ({360/\sides * \i}:2*\radius) circle (2pt);
  }
\end{tikzpicture} 
\end{center}
\caption{The graphs in the first column represent a comb and a radialene. The corresponding conduction graphs are shown in the second column. 
}\label{fig:combsAndRadialenes}
\end{figure}

In chemistry, it is also often useful to know the spectrum of the graphs one is dealing with. For a $\pi$-system, the spectrum is related to electronic structure via the equivalence of the simple Hückel  model~\cite{S61} and the eigenvalue equation for the adjacency matrix of the molecular graph (the graph of the carbon framework). The chemical concepts of molecular orbitals and orbital energies then correspond to eigenvectors and eigenvalues of the adjacency matrix. Many classic results of spectral graph theory have direct consequences for molecular electronic structure in this simplified model. A striking example is given by the Coulson-Rushbrooke Pairing Theorem~\cite{CR40}, which allows prediction of the ground-state $\pi$-electron configuration for every bipartite molecular graph, given its nullity. In the context of Hückel theory and $\pi$-systems, the graphs of interest are the chemical graphs (previously defined as connected graphs with maximum degree at most 3). In other contexts, a definition that allows degree 4 is sometimes used.

For corona graphs, it is possible to describe their spectrum in terms of the spectrum of the base graph:

\begin{theorem}
Let $G=(V_1,E_1)$ be a graph on $n$ vertices with spectrum $\sigma(G)=\{\lambda_1, \lambda_2, ..., \lambda_n\}$ and let $G'=(V_1 \cup V_2, E_2)$ be a corona graph with base graph $G$ (where $|V_2|=n$). We have: 
$$
\sigma(G') = \left\{ \frac{\lambda_i \pm \sqrt{\lambda_i^2 + 4} \vphantom{\frac{\lambda_i^2 + 4}{2}}}{2} \bigm| \lambda_i \in \sigma(G) \right\}.
$$
\end{theorem}
\begin{proof}
Let $x_1, x_2, ..., x_n$ be the eigenvectors of the adjacency matrix of $G$ corresponding to $\lambda_1, \lambda_2, ..., \lambda_n$. We claim that for each $\lambda_i \in \sigma(G)$ there are two eigenvalues $\mu_i, \nu_i := \frac{\lambda_i \pm \sqrt{\lambda_i^2 + 4} \vphantom{\frac{\lambda_i^2 + 4}{2}}}{2}$ in the spectrum of $G'$ with respectively corresponding eigenvectors

    \[
y_{i},z_{i} := \left[
\begin{array}{c}
x_{i,0}\\
x_{i,1}\\
\vdots\\
x_{i,n-1}\\
\frac{2x_{i,0}\vphantom{2x_{i,0}}}{\lambda_i \pm \sqrt{\lambda_i^2+4}}\\
\frac{2x_{i,1}\vphantom{2x_{i,1}}}{\lambda_i \pm \sqrt{\lambda_i^2+4}}\\
\vdots\\
\frac{2x_{i,n-1}\vphantom{2x_{i,n-1}}}{\lambda_i \pm \sqrt{\lambda_i^2+4}}\\
\end{array}
\right].
\]
Indeed, if $G$ has adjacency matrix $A$, then $G'$ has adjacency matrix
    \[
A'=\left[
\renewcommand{\arraystretch}{1.5}
\begin{array}{c|c}
A & I \\ \hline
I & \mathbf{0}
\end{array}
\right].
\]
Therefore, for every vertex $v \in V_1$ we have: 

\begin{align*}
(A'y_{i})_v,(A'z_{i})_v&=(Ax_i)_v+\frac{2x_{i,v}\vphantom{2x_{i,v}}}{\lambda_i \pm \sqrt{\lambda_i^2+4}}\\
&=\lambda_ix_{i,v}+\frac{2x_{i,v}\vphantom{2x_{i,v}}}{\lambda_i \pm \sqrt{\lambda_i^2+4}}\\
&=\lambda_ix_{i,v}+\frac{2x_{i,v}(\lambda_i \mp \sqrt{\lambda_i^2+4})\vphantom{2x_{i,v}(\lambda_i \mp \sqrt{\lambda_i^2+4})}}{\lambda_i^2 - (\lambda_i^2+4)}\\
&=\frac{2\lambda_ix_{i,v}}{2}+\frac{x_{i,v}(\pm \sqrt{\lambda_i^2+4}-\lambda_i)\vphantom{x_{i,v}(\pm \sqrt{\lambda_i^2+4}-\lambda_i)}}{2}\\
&=\frac{\lambda_i \pm \sqrt{\lambda_i^2 + 4} \vphantom{\frac{\lambda_i^2 + 4}{2}}}{2}x_{i,v}\\
&= (\mu_i y_i)_v, (\nu_i z_i)_v.
\end{align*}
Moreover, for every vertex $v \in V_1$, we have:
\begin{align*}
(A'y_{i})_{v+n},(A'z_{i})_{v+n}&=x_{i,v},x_{i,v}\\
&=\frac{\lambda_i \pm \sqrt{\lambda_i^2 + 4} \vphantom{\frac{\lambda_i^2 + 4}{2}}}{2}\frac{2x_{i,v}\vphantom{2x_{i,n-1}}}{\lambda_i \pm \sqrt{\lambda_i^2+4}}\\
&= (\mu_i y_i)_{v+n}, (\nu_i z_i)_{v+n}.
\end{align*}
This proves the claim and completes the proof.
\end{proof}

\subsection{Chemical conduction-isomorphic graphs with minimum degree two}\label{sec:infFam3}
The graphs from Section~\ref{sec:infFam1} all have minimum degree equal to one. In the current section we present an infinite family of chemical conduction-isomorphic graphs with minimum degree two.

\begin{theorem}
\label{th:minDegTwo}
    For each integer $k \geq 2$, there exists a chemical conduction-isomorphic graph with minimum degree equal to 2 on $4k$ vertices.
\end{theorem}
\begin{proof}
For a positive integer $n$ and an integer $a$, we define the matrix $f(n,a)$ as an $n$ by $n$ matrix in which every entry is equal to 0, except one entry per row which is equal to 1: $f(n,a)_{i,i+a}=1$ if $i$ is even and $f(n,a)_{i,i-a}=1$ if $i$ is odd. Here, indices should be interpreted modulo $n$ and the first row (and column) has index 0. For example, the identity matrix with $n$ rows is given by $f(n,0)$. For all integers $k \geq 1$ and $a$, we have:
\begin{align*}
    f(2k,a)^T=\begin{cases}
        f(2k,a) & \textrm{ if }a\textrm{ is odd}, \\
        f(2k,-a) & \textrm{ if }a\textrm{ is even}. \\
    \end{cases}
\end{align*}
Moreover, for all integers $k \geq 1$ and $a, b$, we have:
\begin{align*}
    f(2k,a)f(2k,b)=\begin{cases}
        f(2k,a-b) & \textrm{ if }a\textrm{ is odd}, \\
        f(2k,a+b) & \textrm{ if }a\textrm{ is even}. \\
    \end{cases}
\end{align*}
Let $k \geq 2$ be an integer and let $G$ be the graph with the following adjacency matrix:
    \[
\left[
\renewcommand{\arraystretch}{1.5}
\begin{array}{c|c}
f(2k,1)+f(2k,-1) & f(2k,0) \\ \hline
f(2k,0) & f(2k,1)
\end{array}
\right].
\]
Note that this matrix is indeed symmetric since $f(2k,1)$ and $f(2k,-1)$ are symmetric. By using the definition of $f$ and blockwise matrix multiplication, one can verify that the inverse of this matrix is given by:
    \[
\left[
\renewcommand{\arraystretch}{1.5}
\begin{array}{c|c}
f(2k,-1) & -f(2k,-2) \\ \hline
-f(2k,2) & f(2k,1)+f(2k,3)
\end{array}
\right].
\]
The adjacency matrix of the conduction graph $G\sp{\mathrm C}$ of $G$ is obtained by booleanising the above matrix due to Theorem~\ref{th:booleanise}. This graph and its conduction graph are shown in Fig.~\ref{fig:tFamily} for $k=4$. 
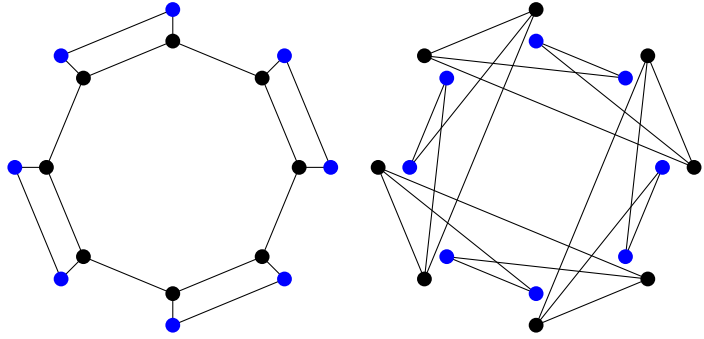
\begin{figure}[h!]
\begin{center}
\begin{tikzpicture}[scale=1.4]
  \def\k{4}
  \pgfmathsetmacro{\sides}{2*\k}
  \def\radius{0.75}

  \foreach \i in {1,...,\sides} {
    \draw ({360/\sides * (\i + 1)}:1.6*\radius) -- ({360/\sides * \i}:1.6*\radius);
  }

  \foreach \i in {1,...,\k} {
    \draw ({360/\sides * (2*\i + 1)}:2*\radius) -- ({360/\sides * (2*\i)}:2*\radius);
  }
  
  \foreach \i in {1,...,\sides} {
    \draw ({360/\sides * (\i)}:1.6*\radius) -- ({360/\sides * \i}:2*\radius);
  }

    \foreach \i in {1,...,\sides} {
    \fill ({360/\sides * \i}:1.6*\radius) circle (2pt);
  }

  \foreach \i in {1,...,\sides} {
    \fill[blue] ({360/\sides * \i}:2*\radius) circle (2pt);
  }
\end{tikzpicture} \quad
\begin{tikzpicture}[scale=1.4]
  \def\k{4}
  \pgfmathsetmacro{\sides}{2*\k}
  \def\radius{0.75}

  \foreach \i in {1,...,\k} {
    \draw ({360/\sides * (2*\i + 1)}:2*\radius) -- ({360/\sides * (2*\i)}:2*\radius);
  }

  \foreach \i in {1,...,\k} {
    \draw ({360/\sides * (2*\i + 3)}:2*\radius) -- ({360/\sides * (2*\i)}:2*\radius);
  }
  
  \foreach \i in {1,...,\k} {
    \draw ({360/\sides * (2*\i)}:2*\radius) -- ({360/\sides * (2*\i+2)}:1.6*\radius);
  }

  \foreach \i in {1,...,\k} {
    \draw ({360/\sides * (2*\i-1)}:2*\radius) -- ({360/\sides * (2*\i-3)}:1.6*\radius);
  }

  \foreach \i in {1,...,\k} {
    \draw ({360/\sides * (2*\i + 2)}:1.6*\radius) -- ({360/\sides * (2*\i+1)}:1.6*\radius);
  }

  \foreach \i in {1,...,\sides} {
    \fill[blue]  ({360/\sides * \i}:1.6*\radius) circle (2pt);
  }

  \foreach \i in {1,...,\sides} {
    \fill ({360/\sides * \i}:2*\radius) circle (2pt);
  }
  
\end{tikzpicture}
\end{center}
\caption{A conduction-isomorphic graph with minimum degree two obtained by choosing $k=4$ (left) and its conduction graph (right).}\label{fig:tFamily}
\end{figure}
Finally, we conclude that $G$ is indeed isomorphic with its conduction graph $G\sp{\mathrm C}$ where the isomorphism $h : V(G) \rightarrow V(G\sp{\mathrm C})$ between $G$ and $G\sp{\mathrm C}$ is given by:
\begin{align*}
    h(u)=\begin{cases}
        u+2k & \textrm{ if }u\textrm{ is even and }0 \leq u \leq 2k-1, \\
        ((u+2) \mod 2k)+2k & \textrm{ if }u\textrm{ is odd and }0 \leq u \leq 2k-1, \\
        ((u+2) \mod 2k) & \textrm{ if }u\textrm{ is even and }2k \leq u \leq 4k-1, \\
        u-2k & \textrm{ if }u\textrm{ is odd and }2k \leq u \leq 4k-1.\\
    \end{cases}
\end{align*}
\end{proof}

\subsection{Conduction-isomorphic graphs with arbitrarily large minimum degree}
We now present an infinite family of conduction-isomorphic graphs with arbitrarily large minimum degree:

\begin{theorem}
\label{th:largeMinDegree}
For each integer $k \geq 3$, there exists a conduction-isomorphic graph on $2k$ vertices with minimum degree equal to $2k-5$.
\end{theorem}
\begin{proof}
    For a positive integer $n$, we define the cyclic permutation matrix $P(n)$ as an $n$ by $n$ matrix in which every entry is equal to 0, except one entry per row which is equal to 1: $P(n)_{i,i-1}=1$ for all $0 \leq i \leq n-1$. Here, indices should be interpreted modulo $n$ and the first row (and column) has index 0. Analogously, the inverse cyclic permutation matrix $P^{-1}(n)$ is an $n$ by $n$ matrix in which every entry is equal to 0, except one entry per row which is equal to 1: $P^{-1}(n)_{i,i+1}=1$.
    
    Let $k \geq 3$ be an integer and let $G$ be the graph with adjacency matrix:
        \[
\left[
\renewcommand{\arraystretch}{1.5}
\begin{array}{c|c}
J-I& J-I-P(k) \\ \hline
J-I-P^{-1}(k) & J-I-P^{-1}(k)-P(k)
\end{array}
\right].
\]
Note that the vertices $0 \leq u \leq k-1$ have degree $2k-3$ and the other vertices have degree $2k-5$. The graph obtained by choosing $k=4$ as well as its conduction graph are shown in Fig.~\ref{fig:largeMinDegree}. 
\begin{figure}[h!]
\begin{center}
\begin{tikzpicture}[scale=1.4]
  \def\k{4}
  \pgfmathsetmacro{\sides}{\k}
  \def\radius{0.75}
  
  \foreach \i in {1,...,\sides} {
    \pgfmathsetmacro{\start}{\i+1}
    \foreach \j in {\start,...,\sides} {
        \draw ({360/\sides * (\i + 0.5)}:1.1*\radius) -- ({360/\sides * (\j+0.5)}:1.1*\radius);
        }
    }

    % draw remaining edges
    % outer to outer
    \draw ({360/\sides * (0)}:2.0*\radius) -- ({360/\sides * (2)}:2.0*\radius);
    \draw ({360/\sides * (1)}:2.0*\radius) -- ({360/\sides * (3)}:2.0*\radius);

    % outer to inner
    \draw ({360/\sides * (0)}:2.0*\radius) -- ({360/\sides * (2+0.5)}:1.1*\radius);
    \draw ({360/\sides * (0)}:2.0*\radius) -- ({360/\sides * (1+0.5)}:1.1*\radius);

    \draw ({360/\sides * (1)}:2.0*\radius) -- ({360/\sides * (3+0.5)}:1.1*\radius);
    \draw ({360/\sides * (1)}:2.0*\radius) -- ({360/\sides * (2+0.5)}:1.1*\radius);

    \draw ({360/\sides * (2)}:2.0*\radius) -- ({360/\sides * (4+0.5)}:1.1*\radius);
    \draw ({360/\sides * (2)}:2.0*\radius) -- ({360/\sides * (3+0.5)}:1.1*\radius);

    \draw ({360/\sides * (3)}:2.0*\radius) -- ({360/\sides * (5+0.5)}:1.1*\radius);
    \draw ({360/\sides * (3)}:2.0*\radius) -- ({360/\sides * (4+0.5)}:1.1*\radius);

      \foreach \i in {1,...,\sides} {
    \fill ({360/\sides * (\i+0.5) }:1.1*\radius) circle (2pt);
  }

  \foreach \i in {1,...,\sides} {
    \fill[blue] ({360/\sides * \i}:2*\radius) circle (2pt);
  }
  
\end{tikzpicture} \quad
\begin{tikzpicture}[scale=1.4]
  \def\k{4}
  \pgfmathsetmacro{\sides}{\k}
  \def\radius{0.75}
  
  \foreach \i in {1,...,\sides} {
    \pgfmathsetmacro{\start}{\i+1}
    \foreach \j in {\start,...,\sides} {
        \draw ({360/\sides * (\i)}:2.0*\radius) -- ({360/\sides * (\j)}:2.0*\radius);
        }
    }

    % draw remaining edges
    % inner to inner
    \draw ({360/\sides * (0+0.5)}:1.1*\radius) -- ({360/\sides * (2+0.5)}:1.1*\radius);
    \draw ({360/\sides * (1+0.5)}:1.1*\radius) -- ({360/\sides * (3+0.5)}:1.1*\radius);

    % outer to inner
    \draw ({360/\sides * (0)}:2.0*\radius) -- ({360/\sides * (2+0.5)}:1.1*\radius);
    \draw ({360/\sides * (0)}:2.0*\radius) -- ({360/\sides * (1+0.5)}:1.1*\radius);

    \draw ({360/\sides * (1)}:2.0*\radius) -- ({360/\sides * (3+0.5)}:1.1*\radius);
    \draw ({360/\sides * (1)}:2.0*\radius) -- ({360/\sides * (2+0.5)}:1.1*\radius);

    \draw ({360/\sides * (2)}:2.0*\radius) -- ({360/\sides * (4+0.5)}:1.1*\radius);
    \draw ({360/\sides * (2)}:2.0*\radius) -- ({360/\sides * (3+0.5)}:1.1*\radius);

    \draw ({360/\sides * (3)}:2.0*\radius) -- ({360/\sides * (5+0.5)}:1.1*\radius);
    \draw ({360/\sides * (3)}:2.0*\radius) -- ({360/\sides * (4+0.5)}:1.1*\radius);

      \foreach \i in {1,...,\sides} {
    \fill[blue]  ({360/\sides * (\i+0.5) }:1.1*\radius) circle (2pt);
  }

  \foreach \i in {1,...,\sides} {
    \fill ({360/\sides * \i}:2*\radius) circle (2pt);
  }
  
\end{tikzpicture}
\end{center}
\caption{A conduction-isomorphic graph with minimum degree three obtained by choosing $k=4$ (left) and its conduction graph (right).}\label{fig:largeMinDegree}
\end{figure}
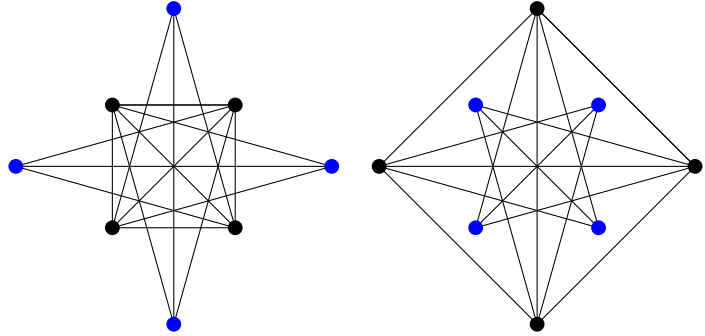

By observing that $P(n)P^{-1}(n)=I_n$ and using blockwise matrix multiplication, we obtain the following inverse matrix:
        \[
\left[
\renewcommand{\arraystretch}{1.5}
\begin{array}{c|c}
-(J-I-P^{-1}(k)-P(k))& J-I-P(k) \\ \hline
J-I-P^{-1}(k) & -(J-I)
\end{array}
\right].
\]
From Theorem~\ref{th:booleanise}, booleanising the above matrix gives us the adjacency matrix of the conduction graph $G\sp{\mathrm C}$ of $G$. We conclude the proof by giving the isomorphism $h : V(G) \rightarrow V(G\sp{\mathrm C})$ between $G$ and $G\sp{\mathrm C}$, which shows that $G$ is indeed conduction-isomorphic:
\begin{align*}
    h(u)=\begin{cases}
        k+((u-1) \mod k)& \textrm{for } 0 \leq u \leq k-1, \\
        u-k & \textrm{for } k \leq u \leq 2k-1. \\
    \end{cases}
\end{align*}
\end{proof}

\section{Constructing conduction-isomorphic graphs from other conduction-isomorphic graphs}\label{sec:condIsoFromCondIso}
Theorem~\ref{th:coronaGraphs} results in infinitely many bipartite and non-bipartite conduction-isomorphic graphs by choosing a suitable base graph and Theorem~\ref{th:largeMinDegree} results in another infinite family of non-bipartite conduction-isomorphic graphs. In the current section, we present another infinite family of such bipartite conduction-isomorphic graphs based on the canonical double cover of a suitable conduction-isomorphic graph. We first recall two well-known observations about such a canonical double cover. The first observation follows directly from the definition of a canonical double cover.
\begin{observation}\label{obs:CDCBipartite}
Let $G$ be any graph. The canonical double cover of $G$ is bipartite.
\end{observation}

The second observation indicates precisely when the canonical double cover of a graph is connected.
\begin{observation}[\cite{BHM80}]\label{obs:CDCConnected}
Let $G$ be any graph. The canonical double cover of $G$ is connected if and only if $G$ is non-bipartite and connected.    
\end{observation}

We now obtain our second infinite family by applying the following theorem to infinitely many connected non-bipartite conduction-isomorphic graphs (e.g.~as obtained from Theorem~\ref{th:largeMinDegree}):
\begin{theorem}
\label{th:canonicalDoubleCover}
Let $G$ be a connected non-bipartite conduction-isomorphic graph. The canonical double cover of $G$ is a connected bipartite conduction-isomorphic graph.
\end{theorem}
\begin{proof}

Let $H$ be the canonical double cover of $G$ with vertex set $V(H) = V(G) \times \{0,1\}$ and edge set $E(H) = \{(u_1,j)(u_2,1-j)~|~u_1u_2 \in E(G), j \in \{0,1\}\}$. Because of Observations~\ref{obs:CDCBipartite} and~\ref{obs:CDCConnected}, $H$ is indeed a connected bipartite graph. Let $A$ be the adjacency matrix of $G$ and $h' : V(G) \rightarrow V(G\sp{\mathrm C})$ be an isomorphism between $G$ and its conduction graph $G\sp{\mathrm C}$. The graph $H$ has the following adjacency matrix:
    \[
\left[
\renewcommand{\arraystretch}{1.5}
\begin{array}{c|c}
\mathbf{0} & A \\ \hline
A^T & \mathbf{0}
\end{array}
\right].
\]

Using blockwise matrix multiplication and the fact that $A=A^T$, we obtain the following inverse:
    \[
\left[
\renewcommand{\arraystretch}{1.5}
\begin{array}{c|c}
\mathbf{0} & (A^T)^{-1} \\ \hline
A^{-1} & \mathbf{0}
\end{array}
\right].
\]
Because of Theorem~\ref{th:booleanise}, the conduction graph $H^C$ of $H$ has an adjacency matrix which is obtained by booleanising the above inverse. As an example, Fig.~\ref{fig:nonBipartite} shows the smallest connected non-bipartite conduction-isomorphic graph as well as its canonical double cover.

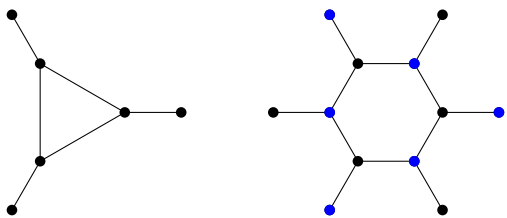
\begin{figure}[h!]
\begin{center}
\begin{tikzpicture}[scale=1.0]
  \def\sides{3}
  \def\radius{0.75}
  
  \foreach \i in {1,...,\sides} {
    \draw ({360/\sides * (\i + 1)}:\radius) -- ({360/\sides * \i}:\radius);
  }

  \foreach \i in {1,...,\sides} {
    \draw ({360/\sides * (\i)}:\radius) -- ({360/\sides * \i}:2*\radius);
  }

  \foreach \i in {1,...,\sides} {
    \fill ({360/\sides * \i}:\radius) circle (2pt);
  }

  \foreach \i in {1,...,\sides} {
    \fill[black] ({360/\sides * \i}:2*\radius) circle (2pt);
  }
  
\end{tikzpicture} \quad\quad\quad
\begin{tikzpicture}[scale=1.0]
  \def\sides{6}
  \def\radius{0.75}
  
  \foreach \i in {1,...,\sides} {
    \draw ({360/\sides * (\i + 1)}:\radius) -- ({360/\sides * \i}:\radius);
  }

  \foreach \i in {1,...,\sides} {
    \draw ({360/\sides * (\i)}:\radius) -- ({360/\sides * \i}:2*\radius);
  }

  \foreach \i in {1,...,\sides} {
    \fill ({360/\sides * \i}:\radius) circle (2pt);
  }
  \foreach \i in {1,3,...,\sides} {
    \fill[blue] ({360/\sides * \i}:\radius) circle (2pt);
  }
  
  \foreach \i in {1,...,\sides} {
    \fill[black] ({360/\sides * \i}:2*\radius) circle (2pt);
  }

  \foreach \i in {2,4,...,\sides} {
    \fill[blue] ({360/\sides * \i}:2*\radius) circle (2pt);
  }
  
\end{tikzpicture}
\end{center}
\caption{The smallest connected non-bipartite conduction-isomorphic graph (left) and its canonical double cover (right).}\label{fig:nonBipartite}
\end{figure}

Now $H$ is conduction-isomorphic where the isomorphism $h : V(H) \rightarrow V(H^\mathrm{C})$ between $H$ and $H^\mathrm{C}$ is given by $h((u,j))=(h'(u),j)$ (for each $u \in V(G)$, $j \in \{0,1\}$).
\end{proof}

We remark that the canonical double cover of a graph $G$ has the same degree set as $G$ itself, and therefore the canonical double cover of a non-bipartite chemical conduction-isomorphic graph is again a chemical conduction-isomorphic graph.

\section{Nonexistence of 3-regular conduction-isomorphic graphs}
\label{sec:no3Regular}
Among others, the results of Section~\ref{sec:infFam} imply the existence of infinitely many chemical conduction-isomorphic graphs with minimum degree 1 and 2. In the current section we show that there are no chemical conduction-isomorphic graphs with minimum degree 3, i.e.\ simple connected 3-regular conduction-isomorphic graphs. This rules out the existence of conduction-isomorphic graphs among a number of important chemical graph families (e.g.\ fullerenes and other carbon cage molecules based on cubic polyhedra).

\begin{theorem}
    Let $G$ be a simple connected 3-regular graph with nullity 0 and having the conduction graph $G\sp{\mathrm C}$. Then every vertex $v \in V(G\sp{\mathrm C})$ has at least 4 neighbours. As a result, $G$ is not conduction-isomorphic.
\end{theorem}
\begin{proof}
    Let $A$ be the adjacency matrix of $G$ and suppose for the sake of obtaining a contradiction that there exists a vertex $v \in V(G\sp{\mathrm C})$ which has strictly less than 4 neighbours. Because of Theorem~\ref{th:booleanise} this implies that there exists a column $c$ in $A^{-1}$ which has precisely $k<4$ non-zero elements $a_1, ..., a_k$. Since $AA^{-1}=I$, the product $Ac$ has to contain precisely one entry equal to 1 and all other entries equal to 0. This means that there exists a family $\mathcal{F}$ which is a multiset of subsets of $\{1,...,k\}$ such that there is precisely one subset $S \in \mathcal{F}$ for which $\sum_{i \in S}a_i=1$ and for all other subsets $S' \in \mathcal{F}$ we have $\sum_{i \in S'}a_i=0$. Moreover, each $i \in \{1,...,k\}$ appears in exactly 3 subsets of $\mathcal{F}$, since $G$ is 3-regular. Additionally, for $S' \in \mathcal{F}$ and $S' \neq S$, we have $|S'| \geq 2$, because the elements $a_1, ..., a_k$ are non-zero. Due to symmetry, we may assume without loss of generality that $S=\{1,...,l\}$ for some integer $1 \leq l \leq k$. Hence, there are only 6 such families $\mathcal{F}$ giving rise to the following sets of linear equations:

        \[
    \left\{
    \begin{aligned}
        a_1 + a_2 &= 1 \\
        a_1 + a_2 &= 0 \\
        a_1 + a_2 &= 0
    \end{aligned}
    \right.
    \quad
    \left\{
    \begin{aligned}
        a_1 &= 1 \\
        a_1 + a_2 &= 0 \\
        a_1 + a_3 &= 0 \\
        a_2 + a_3 &= 0 \\
        a_2 + a_3 &= 0 \\
    \end{aligned}
    \right.
    \quad
    \left\{
    \begin{aligned}
        a_1 &= 1 \\
        a_2 + a_3 &= 0 \\
        a_1 + a_2 + a_3 &= 0 \\
        a_1 + a_2 + a_3 &= 0 \\
    \end{aligned}
    \right.
    \]

            \[
    \left\{
    \begin{aligned}
        a_1 + a_2 &= 1 \\
        a_1 + a_3 &= 0 \\
        a_2 + a_3 &= 0 \\
        a_1 + a_2 + a_3 &= 0 \\
    \end{aligned}
    \right.
    \quad
    \left\{
    \begin{aligned}
        a_1 + a_2 + a_3 &= 1 \\
        a_1 + a_2 &= 0 \\
        a_1 + a_3 &= 0 \\
        a_2 + a_3 &= 0 \\
    \end{aligned}
    \right.
    \quad
    \left\{
    \begin{aligned}
        a_1 + a_2 + a_3 &= 1 \\
        a_1 + a_2 + a_3 &= 0 \\
        a_1 + a_2 + a_3 &= 0 \\
    \end{aligned}
    \right.
    \]
    
    None of these sets of linear equations has a solution and therefore we obtain a contradiction, as desired.
\end{proof}

We remark that the 3-regular case is different from the $d$-regular case for $d>3$. For example, Fig.~\ref{fig:6Reg} shows a 6-regular graph for which its conduction graph has a vertex of degree 6. However, this graph is not conduction-isomorphic and we were unable to find any such graphs. We therefore ask:
\begin{question}
    Does there exist a $d$-regular conduction-isomorphic graph for some integer $d>3$?
\end{question}

% W???r?WEB?K?BRBR@heC?oCB?_?i?wgOyao_gkODDc?A_b_
% dWDZ?~?[FF?[?F?F?B_FD_wkBao?yW?yWB_B_w?wF?F?Bh_?

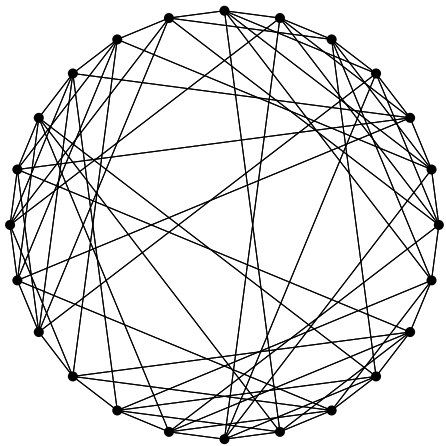
\begin{figure}[h!]
\begin{center}
\begin{tikzpicture}[scale=0.95]
  \def\sides{24}
  \def\radius{3}

  \foreach \i in {1,...,\sides} {
    \fill ({360/\sides * \i}:\radius) circle (2pt);
  }
  
  % Draw the 24-gon
  \foreach \i in {1,9,17} {
    \draw ({360/\sides * (\i + 1)}:\radius) -- ({360/\sides * \i}:\radius);
    \draw ({360/\sides * (\i + 3)}:\radius) -- ({360/\sides * \i}:\radius);
    \draw ({360/\sides * (\i + 5)}:\radius) -- ({360/\sides * \i}:\radius);
    \draw ({360/\sides * (\i + 7)}:\radius) -- ({360/\sides * \i}:\radius);
    \draw ({360/\sides * (\i + 17)}:\radius) -- ({360/\sides * \i}:\radius);
    \draw ({360/\sides * (\i + 23)}:\radius) -- ({360/\sides * \i}:\radius);
  }
    \foreach \i in {2,10,18} {
    \draw ({360/\sides * (\i + 1)}:\radius) -- ({360/\sides * \i}:\radius);
    \draw ({360/\sides * (\i + 3)}:\radius) -- ({360/\sides * \i}:\radius);
    \draw ({360/\sides * (\i + 7)}:\radius) -- ({360/\sides * \i}:\radius);
    \draw ({360/\sides * (\i + 9)}:\radius) -- ({360/\sides * \i}:\radius);
    \draw ({360/\sides * (\i + 11)}:\radius) -- ({360/\sides * \i}:\radius);
    \draw ({360/\sides * (\i + 23)}:\radius) -- ({360/\sides * \i}:\radius);
  }
  \foreach \i in {3,11,19} {
    \draw ({360/\sides * (\i + 1)}:\radius) -- ({360/\sides * \i}:\radius);
    \draw ({360/\sides * (\i + 3)}:\radius) -- ({360/\sides * \i}:\radius);
    \draw ({360/\sides * (\i + 11)}:\radius) -- ({360/\sides * \i}:\radius);
    \draw ({360/\sides * (\i + 15)}:\radius) -- ({360/\sides * \i}:\radius);
    \draw ({360/\sides * (\i + 21)}:\radius) -- ({360/\sides * \i}:\radius);
    \draw ({360/\sides * (\i + 23)}:\radius) -- ({360/\sides * \i}:\radius);
  }
  \foreach \i in {4,12,20} {
    \draw ({360/\sides * (\i + 1)}:\radius) -- ({360/\sides * \i}:\radius);
    \draw ({360/\sides * (\i + 3)}:\radius) -- ({360/\sides * \i}:\radius);
    \draw ({360/\sides * (\i + 17)}:\radius) -- ({360/\sides * \i}:\radius);
    \draw ({360/\sides * (\i + 19)}:\radius) -- ({360/\sides * \i}:\radius);
    \draw ({360/\sides * (\i + 21)}:\radius) -- ({360/\sides * \i}:\radius);
    \draw ({360/\sides * (\i + 23)}:\radius) -- ({360/\sides * \i}:\radius);
  }
  \foreach \i in {5,13,21} {
    \draw ({360/\sides * (\i + 1)}:\radius) -- ({360/\sides * \i}:\radius);
    \draw ({360/\sides * (\i + 7)}:\radius) -- ({360/\sides * \i}:\radius);
    \draw ({360/\sides * (\i + 13)}:\radius) -- ({360/\sides * \i}:\radius);
    \draw ({360/\sides * (\i + 19)}:\radius) -- ({360/\sides * \i}:\radius);
    \draw ({360/\sides * (\i + 21)}:\radius) -- ({360/\sides * \i}:\radius);
    \draw ({360/\sides * (\i + 23)}:\radius) -- ({360/\sides * \i}:\radius);
  }
  \foreach \i in {6,14,22} {
    \draw ({360/\sides * (\i + 1)}:\radius) -- ({360/\sides * \i}:\radius);
    \draw ({360/\sides * (\i + 13)}:\radius) -- ({360/\sides * \i}:\radius);
    \draw ({360/\sides * (\i + 17)}:\radius) -- ({360/\sides * \i}:\radius);
    \draw ({360/\sides * (\i + 19)}:\radius) -- ({360/\sides * \i}:\radius);
    \draw ({360/\sides * (\i + 21)}:\radius) -- ({360/\sides * \i}:\radius);
    \draw ({360/\sides * (\i + 23)}:\radius) -- ({360/\sides * \i}:\radius);
  }
  \foreach \i in {7,15,23} {
    \draw ({360/\sides * (\i + 1)}:\radius) -- ({360/\sides * \i}:\radius);
    \draw ({360/\sides * (\i + 5)}:\radius) -- ({360/\sides * \i}:\radius);
    \draw ({360/\sides * (\i + 7)}:\radius) -- ({360/\sides * \i}:\radius);
    \draw ({360/\sides * (\i + 17)}:\radius) -- ({360/\sides * \i}:\radius);
    \draw ({360/\sides * (\i + 21)}:\radius) -- ({360/\sides * \i}:\radius);
    \draw ({360/\sides * (\i + 23)}:\radius) -- ({360/\sides * \i}:\radius);
  }
  \foreach \i in {8,16,24} {
    \draw ({360/\sides * (\i + 1)}:\radius) -- ({360/\sides * \i}:\radius);
    \draw ({360/\sides * (\i + 3)}:\radius) -- ({360/\sides * \i}:\radius);
    \draw ({360/\sides * (\i + 5)}:\radius) -- ({360/\sides * \i}:\radius);
    \draw ({360/\sides * (\i + 7)}:\radius) -- ({360/\sides * \i}:\radius);
    \draw ({360/\sides * (\i + 17)}:\radius) -- ({360/\sides * \i}:\radius);
    \draw ({360/\sides * (\i + 23)}:\radius) -- ({360/\sides * \i}:\radius);
  }
  
\end{tikzpicture}
\end{center}
\caption{A 6-regular graph on 24 vertices whose conduction graph contains 18 vertices of degree 6.}\label{fig:6Reg}
\end{figure}
\section{Computational results}
\label{sec:computationalResults}
We used the program \texttt{nauty}~\cite{MP14} for generating all pairwise non-isomorphic connected graphs up to 11 vertices and chemical graphs up to 22 vertices. Using a supercomputer, a program to calculate the inverse of a matrix and the program \texttt{nauty} for quickly verifying whether two graphs are isomorphic, we determined all pairwise non-isomorphic conduction-isomorphic graphs within these two classes. We also made these graphs available at the House of Graphs~\cite{CDG23} where they can be obtained by searching for the term “conduction-isomorphic”. In view of Theorem~\ref{th:canonicalDoubleCover}, we also determined which of them are non-bipartite graphs. In total, these computations took around 6 CPU-months. In Tables~\ref{tab:countsConnected} and~\ref{tab:countsChemical}, we summarize the counts of the graphs that we considered\footnote{We remark that rounding errors are possible when computing the inverse of a matrix 
in real arithmetic
and therefore the correctness of the counts is also subject to these rounding errors. However, since all considered graphs are small, we believe that the rounding errors are never large, 
and that
therefore the counts are correct. To check this intuition,
we considered numbers whose absolute value is at most $\epsilon$ to be equal to 0; for each graph, we 
performed three independent runs using $\epsilon=10^{-7}$, $\epsilon=10^{-9}$ and $\epsilon=10^{-11}$; each run gave the same result in terms of graph counts. Difficult cases could be triaged and flagged for re-investigation with  
higher precision, but so far this has not proved necessary.}.

\begin{table}[htbp]
    \centering
    \footnotesize
    \begin{tabular}{|c| c | c | c |}
        \hline
        $n$ & \# connected & \thead{\# connected\\conduction-isomorphic} & \thead{\# connected\\conduction-isomorphic\\non-bipartite}\\
        \hline
        1 & 1 & 0 & 0\\
        2 & 1 & 1 & 0\\
        3 & 2 & 0 & 0\\
        4 & 6 & 1 & 0\\
        5 & 21 & 0 & 0\\
        6 & 112 & 4 & 2\\
        7 & 853 & 0 & 0\\
        8 & 11 117 & 33 & 24\\
        9 & 261 080 & 0 & 0\\
        10 & 11 716 571 & 358 & 332\\
        11 & 1 006 700 565 & 0 & 0\\
        \hline
    \end{tabular}
    \vskip12pt
    \caption{An overview of the number of pairwise non-isomorphic connected graphs, connected conduction-isomorphic graphs and connected conduction-isomorphic non-bipartite graphs on $n$ vertices.}
    \label{tab:countsConnected}
\end{table}

%\begin{table}[htbp]
\begin{table}[H]
        \centering
\begin{tabular}{|c| c | c | c |}
        \hline
        $n$ & \# chemical & \thead{\# chemical\\conduction-isomorphic} & \thead{\# chemical\\conduction-isomorphic\\non-bipartite}\\
        \hline
        1 & 1 & 0 & 0\\
        2 & 1 & 1 & 0\\
        3 & 2 & 0 & 0\\
        4 & 6 & 1 & 0\\
        5 & 10 & 0 & 0\\
        6 & 29 & 3 & 1\\
        7 & 64 & 0 & 0\\
        8 & 194 & 5 & 0\\
        9 & 531 & 0 & 0\\
        10 & 1733 & 3 & 1\\
        11 & 5524 & 0 & 0\\
        12 & 19 430 & 4 & 0\\
        13 & 69 322 & 0 & 0\\
        14 & 262 044 & 2 & 1\\
        15 & 1 016 740 & 0 & 0\\
        16 & 4 101 318 & 4 & 0\\
        17 & 16 996 157 & 0 & 0\\
        18 & 72 556 640 & 2 & 1\\
        19 & 317 558 689 & 0 & 0\\
        20 & 1 424 644 848 & 4 & 0\\
        21 & 6 536 588 420 & 0 & 0\\
        22 & 30 647 561 117 & 2 & 1\\
        \hline
    \end{tabular}
     \vskip12pt
    \caption{An overview of the number of pairwise non-isomorphic chemical graphs, chemical conduction-isomorphic graphs and chemical conduction-isomorphic non-bipartite graphs on $n$ vertices.}
    \label{tab:countsChemical}
\end{table}

These counts suggest that only a small fraction of all connected graphs on a fixed number of vertices $n$ are conduction-isomorphic and that an even smaller fraction of such graphs are also chemical graphs. By inspecting these graphs more carefully, we also discovered another infinite family of chemical conduction-isomorphic graphs. This new family of graphs is more cumbersome to describe than the families we considered before and we refer the interested reader to the Appendix for their description and a proof that they are indeed conduction-isomorphic. Interestingly, there are only 3 chemical graphs on at most 22 vertices which do not belong to any of the infinite families of graphs that are described in the current paper (shown in Fig.~\ref{fig:exceptions}).

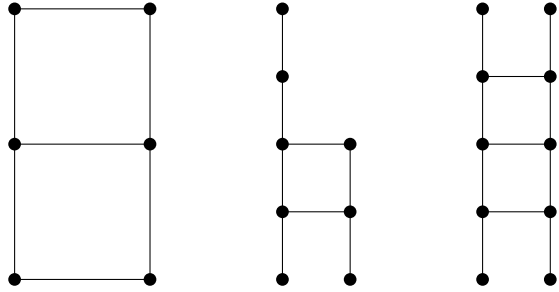
\begin{figure}[h!]
\begin{center}
\begin{tikzpicture}[scale=1.2]
  \def\inter{1.5}
  \def\k{3}

    \foreach \i in {1,...,\k}
    {
        \draw ({\inter * 1}, {\inter * (\i)}) -- ({\inter * 2}, {\inter * (\i)});
    }
    \draw ({\inter * 1}, {\inter * (1)}) -- ({\inter * 1}, {\inter * (2)});
    \draw ({\inter * 1}, {\inter * (2)}) -- ({\inter * 1}, {\inter * (3)});
    \draw ({\inter * 2}, {\inter * (1)}) -- ({\inter * 2}, {\inter * (2)});
    \draw ({\inter * 2}, {\inter * (2)}) -- ({\inter * 2}, {\inter * (3)});
    
      \foreach \i in {1,...,\k}
    {
        \fill (\inter * 1,\inter * \i) circle (2pt);
    }

    \foreach \i in {1,...,\k}
    {
        \fill(\inter * 2,\inter * \i) circle (2pt);
    }
  
\end{tikzpicture} \quad\quad\quad
\hspace{0.4cm}
\begin{tikzpicture}[scale=1.2]
  \def\inter{0.75}
  \def\k{5}

    \draw ({\inter * 1}, {\inter * (1)}) -- ({\inter * 1}, {\inter * (2)});
    \draw ({\inter * 1}, {\inter * (2)}) -- ({\inter * 1}, {\inter * (3)});
    \draw ({\inter * 1}, {\inter * (3)}) -- ({\inter * 1}, {\inter * (4)});
    \draw ({\inter * 1}, {\inter * (4)}) -- ({\inter * 1}, {\inter * (5)});
    
    \draw ({\inter * 2}, {\inter * (1)}) -- ({\inter * 2}, {\inter * (2)});
    \draw ({\inter * 2}, {\inter * (2)}) -- ({\inter * 2}, {\inter * (3)});

    \draw ({\inter * 1}, {\inter * (2)}) -- ({\inter * 2}, {\inter * (2)});
    \draw ({\inter * 1}, {\inter * (3)}) -- ({\inter * 2}, {\inter * (3)});
    
      \foreach \i in {1,...,\k}
    {
        \fill (\inter * 1,\inter * \i) circle (2pt);
    }

    \foreach \i in {1,...,3}
    {
        \fill(\inter * 2,\inter * \i) circle (2pt);
    }
  
\end{tikzpicture} \quad\quad\quad
\hspace{0.4cm}
\begin{tikzpicture}[scale=1.2]
  \def\inter{0.75}
  \def\k{5}

    \draw ({\inter * 1}, {\inter * (1)}) -- ({\inter * 1}, {\inter * (2)});
    \draw ({\inter * 1}, {\inter * (2)}) -- ({\inter * 1}, {\inter * (3)});
    \draw ({\inter * 1}, {\inter * (3)}) -- ({\inter * 1}, {\inter * (4)});
    \draw ({\inter * 1}, {\inter * (4)}) -- ({\inter * 1}, {\inter * (5)});
    
    \draw ({\inter * 2}, {\inter * (1)}) -- ({\inter * 2}, {\inter * (2)});
    \draw ({\inter * 2}, {\inter * (2)}) -- ({\inter * 2}, {\inter * (3)});
    \draw ({\inter * 2}, {\inter * (3)}) -- ({\inter * 2}, {\inter * (4)});
    \draw ({\inter * 2}, {\inter * (4)}) -- ({\inter * 2}, {\inter * (5)});
    
    \draw ({\inter * 1}, {\inter * (2)}) -- ({\inter * 2}, {\inter * (2)});
    \draw ({\inter * 1}, {\inter * (3)}) -- ({\inter * 2}, {\inter * (3)});
    \draw ({\inter * 1}, {\inter * (4)}) -- ({\inter * 2}, {\inter * (4)});
    
      \foreach \i in {1,...,\k}
    {
        \fill (\inter * 1,\inter * \i) circle (2pt);
    }

    \foreach \i in {1,...,\k}
    {
        \fill(\inter * 2,\inter * \i) circle (2pt);
    }
  
\end{tikzpicture}
\end{center}
\caption{Three chemical conduction-isomorphic graphs.}\label{fig:exceptions}
\end{figure}
% E`dg
% G@AAWw
% I??KQKSIG

Moreover, a clear pattern in the counts seems to emerge starting from $n=11$ vertices (there are 0, 4, 0, and 2 chemical conduction-isomorphic graphs on, respectively, $4k-1$, $4k$, $4k+1$ and $4k+2$ vertices for $k \geq 3$). These results seem to suggest that perhaps the infinite families described in the current paper cover (nearly) all cases of chemical conduction-isomorphic graphs. Based on this, we ask the following question:
\begin{question}
Are there only a finite number of chemical conduction-isomorphic graphs which do not belong to one of the infinite families described in Theorems~\ref{th:coronaGraphs},~\ref{th:minDegTwo} and~\ref{th:fourthInfFam}?
\end{question}

\section{Conclusion}
\label{sec:conclusion}
The conduction graph, $G\sp{\mathrm C}$, has been defined here as a compact way of summarising at a glance the Fermi-level conduction behaviour of all possible devices derived by connection of a simple graph $G$ to two leads. 

The graph $G\sp{\mathrm C}$ for given $G$ is derived from the conduction matrix for all devices (here calculated within the SSP model~\cite{E11,GEZ07}) by a ‘booleanisation’ process.   Edges and loops in $G\sp{\mathrm C}$ denote conducting distinct and ipso devices, respectively. 

Properties of $G\sp{\mathrm C}$ are shown to link to previous work 
by the Sheffield group and collaborators within the SSP model in which selection rules  
and classification schemes for conduction classes~\cite{FBPS20,FPTBS14,FSBSP17} were derived. 
In particular, the case where 
$G\sp{\mathrm C}$ is isomorphic to the simple graph $G$ is completely described here: $G$ is non-singular (Corollary~\ref{cor:condIsoNullity0}), and is an ipso omni-insulator (Observation~\ref{obs:condIsoIsIpsoOmniInsulator}).

Constructions of infinite families of conduction-isomorphic graphs, some of which are chemical, have been found.

Investigation of conduction graphs derived from singular graphs  gives a characterisation of the 
well studied class of nut graphs as those graphs of nullity one for which 
$G\sp{\mathrm C}$ is connected. 

Generalisations to graphs of higher nullity, and in particular to the core graphs with nullity greater than one,
could be interesting, and 
questions are scattered throughout the text in the hope of stimulating 
further work.

\section*{Acknowledgements}
\noindent The computational resources and services used in this work were provided by the VSC (Flemish Supercomputer Centre), funded by the Research Foundation Flanders (FWO) and the Flemish Government - Department EWI. The research of Jan Goedgebeur was supported by Internal Funds of KU Leuven and an FWO grant with grant number G0AGX24N. Jorik Jooken is supported by a Postdoctoral Fellowship of the Research Foundation Flanders (FWO) with grant number 1222524N. Patrick Fowler thanks the Francqui Foundation
for the award of an International Francqui Professor Chair, which supported an extended research visit to Belgium in 2024. The authors are also grateful to Stijn Cambie for suggestions that helped to improve the presentation of this manuscript.

\newpage
\section*{Appendix}\label{sec: appendix}

\appendix
\section*{Another family of chemical conduction-isomorphic graphs}
\begin{theorem}
\label{th:fourthInfFam}
For all integers $k \geq 3$, there exists a chemical conduction-isomorphic graph with minimum degree one on $4k-4$ vertices, different from the graphs described in Theorem~\ref{th:coronaGraphs}.
\end{theorem}
\begin{proof}
We shall use again the matrix $f(n,a)$ as introduced in Theorem~\ref{th:minDegTwo}, i.e.\ $f(n,a)$ is an $n$ by $n$ matrix in which every entry is equal to 0, except one entry per row which is equal to 1: $f(n,a)_{i,i+a}=1$ if $i$ is even and $f(n,a)_{i,i-a}=1$ if $i$ is odd. For positive integers $n_1$ and $n_2$ and non-negative integers $a$ and $b$, we use $E^{n_1,n_2}_{a,b}$ to denote an $n_1$ by $n_2$ matrix in which every entry is equal to 0, except for the entry on row $a$ and column $b$, which is equal to 1. Let $k \geq 3$ be an integer and let $G$ be a graph on $4k-4$ vertices having the following adjacency matrix:

    \[
\left[
\renewcommand{\arraystretch}{1.5}
\begin{array}{c|c}
f(2k,1)+f(2k,-1)&\\-
E^{2k,2k}_{0,2k-1}& A\\
-E^{2k,2k}_{2k-1,0} & \\ \hline
A^T & f(2k-4,1)
\end{array}
\right].
\]
Here, the matrix $A$ is a matrix with $2k$ rows and $2k-4$ columns in which the first three rows and the last row are filled with zeroes and the other $2k-4$ rows from top to bottom are given by the identity matrix $I_{2k-4}$. By blockwise matrix multiplication, we can verify that the inverse matrix is given by:

    \[
\left[
\renewcommand{\arraystretch}{1.5}
\begin{array}{c|c}
f(2k,1)&\\
-E^{2k,2k}_{0,3}& E^{2k,2k-4}_{0,1}-B-C \\
-E^{2k,2k}_{3,0}&\\ \hline
 & f(2k-4,1)+f(2k-4,3)\\
(E^{2k,2k-4}_{0,1}-B-C)^T & -E^{2k-4,2k-4}_{1,2k-5}\\
& -E^{2k-4,2k-4}_{2k-5,1}
\end{array}
\right].
\]
Here, $B$ is a matrix consisting of $2k$ rows and $2k-4$ columns filled with zeroes except $B_{2,1}=B_{4,3}=...=B_{2k-4,2k-5}=1$. Likewise, $C$ is a matrix consisting of $2k$ rows and $2k-4$ columns filled with zeroes except $C_{5,0}=C_{7,2}=...=C_{2k-1,2k-6}=1$. Because of Theorem~\ref{th:booleanise}, booleanising the above inverse matrix yields the adjacency matrix of the conduction graph $G\sp{\mathrm C}$ of $G$. We give an example of the graphs obtained by choosing $k=4$ in Fig.~\ref{fig:fourthInfFam}.

\begin{figure}[h!]
\begin{center}
\begin{tikzpicture}[scale=1.7]
  \def\inter{0.45}
  \def\k{8}

  \pgfmathsetmacro{\kMinusOne}{\k-1}

    \fill[white] (\inter * 3,\inter * 2.9) circle (2pt);
    \fill[white] (\inter * 3,\inter * 0.15) circle (2pt);
    
    \foreach \i in {2,...,\k}
    {
        \draw ({\inter * (\i-1)}, \inter * 1) -- ({\inter * (\i)}, \inter * 1);
    }

    \foreach \i in {4,...,\kMinusOne}
    {
        \draw ({\inter * (\i)}, \inter * 1) -- ({\inter * (\i)}, \inter * 2);
    }

    \foreach \i in {4,6,...,\kMinusOne}
    {
        \draw ({\inter * (\i+1)}, \inter * 2) -- ({\inter * (\i)}, \inter * 2);
    }

      \foreach \i in {1,...,\k}
    {
        \fill (\inter * \i,\inter * 1) circle (2pt);
    }
    
    \foreach \i in {4,...,\kMinusOne}
    {
        \fill[blue] (\inter * \i,\inter * 2) circle (2pt);
    }
    %\fill[pink] (\inter * 1,\inter * 1) circle (2pt);
  
\end{tikzpicture} \quad
\vspace{0.3cm}
\begin{tikzpicture}[scale=1.7]
  \def\inter{0.45}
  \def\k{8}

  \pgfmathsetmacro{\kMinusOne}{\k-1}
  \pgfmathsetmacro{\kMinusTwo}{\k-2}

    \foreach \i in {2,4,...,\k}
    {
        \draw ({\inter * (\i-1)}, \inter * 1) -- ({\inter * (\i)}, \inter * 1);
    }

     \draw (\inter * 4,\inter * 2) .. controls (\inter * 5.5,\inter * 3) and (\inter * 5.5,\inter * 3) .. (\inter * 7,\inter * 2);
     \draw (\inter * 1,\inter * 1) .. controls (\inter * 2.5,\inter * 0) and (\inter * 2.5,\inter * 0) .. (\inter * 4,\inter * 1);
     \draw (\inter * 1,\inter * 1) .. controls (\inter * 2.5,\inter * 2.7) and (\inter * 2.5,\inter * 2.7) .. (\inter * 5,\inter * 2);

    \draw ({\inter * (3)}, \inter * 1) -- ({\inter * (5)}, \inter * 2);
     \draw ({\inter * (6)}, \inter * 1) -- ({\inter * (4)}, \inter * 2);
     \draw ({\inter * (5)}, \inter * 1) -- ({\inter * (7)}, \inter * 2);
     \draw ({\inter * (8)}, \inter * 1) -- ({\inter * (6)}, \inter * 2);
     
    \foreach \i in {4,6,...,\kMinusOne}
    {
        \draw ({\inter * (\i+1)}, \inter * 2) -- ({\inter * (\i)}, \inter * 2);
    }
    
    \foreach \i in {1,...,\k}
    {
        \fill (\inter * \i,\inter * 1) circle (2pt);
    }

        \foreach \i in {3,...,6}
    {
        \fill[blue] (\inter * \i,\inter * 1) circle (2pt);
    }
    %\fill[pink] (\inter * 7,\inter * 1) circle (2pt);
    
    \foreach \i in {4,...,\kMinusOne}
    {
        \fill (\inter * \i,\inter * 2) circle (2pt);
    }
\end{tikzpicture}
\end{center}
\caption{The graph $G$ (left) and its conduction graph $G\sp{\mathrm C}$ (right) obtained by choosing $k=4$ in Theorem~\ref{th:fourthInfFam}.}\label{fig:fourthInfFam}
\end{figure}
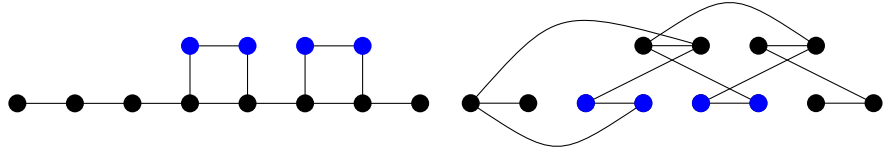

Finally, we give the isomorphism $h : V(G) \rightarrow V(G\sp{\mathrm C})$ between $G$ and $G\sp{\mathrm C}$ that is necessary to be able to conclude that $G$ is indeed conduction-isomorphic:
\begin{align*}
    h(u)=\begin{cases}
        2k-2+u& \textrm{for } u \in \{0,1\}, \\
        2k-2-u& \textrm{for } u \in \{2k-2,2k-1\}, \\
        4k-4-u& \textrm{for } u \notin \{0,2k-2\}\text{ and }u\textrm{ even}, \\
        4k-2-u& \textrm{for } u \notin \{1,2k-1\}\text{ and }u\textrm{ odd}. \\
    \end{cases}
\end{align*}

\end{proof}

\end{document}